\documentclass[reqno]{amsart}
\usepackage[foot]{amsaddr}
\usepackage[utf8]{inputenc}
\usepackage{amsmath}
\usepackage{amsfonts}
\usepackage{layout}
\usepackage{amssymb}
\usepackage{csquotes}
\usepackage{graphicx} 
\usepackage{bbm}

\usepackage{biblatex}
\usepackage{xcolor}

\usepackage{geometry}

\usepackage[colorlinks]{hyperref}

\theoremstyle{plain}
\newtheorem{theorem}{Theorem}[section]

\newtheorem{lemma}[theorem]{Lemma}
\newtheorem*{lemma*}{Lemma}

\newtheorem{remark}[theorem]{Remark}
\DeclareMathOperator{\supp}{supp}

\newtheorem*{assumption*}{\assumptionnumber}
\providecommand{\assumptionnumber}{}
\makeatletter
\newenvironment{assumption}[2]
 {%
  \renewcommand{\assumptionnumber}{Assumption #1}
  \begin{assumption*}%
  \protected@edef\@currentlabel{#1}%
 }
 {%
  \end{assumption*}
 }
\makeatother

\author{Matteo Cagnotti}
\address{Department of Mathematics “G. Peano”, University of Turin, Via Carlo Alberto 10, 10123 Torino, Italy.}
\email{matteo.cagnotti@unito.it}
\title[\resizebox{4.5in}{!}{$L^p$-sup Convergence of the Euler-Maruyama Scheme for SDEs with Distributional Besov Drift}]{$L^p$-sup Convergence of the Euler-Maruyama Scheme for SDEs with Distributional Besov Drift}
\begin{document}
\begin{abstract}
In this paper we extend existing results on the numerical approximation of one-dimensional SDEs with drift in a negative order Besov space and driven by Brownian motion. Using the Yamada-Watanabe approximation technique, we prove rates in $L^p$, for all $p\geq 2$, applying a Gronwall-type lemma previously used in the literature for SDEs with H\"older continuous coefficients. Additionally, we obtain an explicit convergence rate in the $L^1$-$\sup$ norm.
\end{abstract}
\maketitle
\section{Introduction and Main Results}
Stochastic differential equations (SDEs) with irregular or distributional drift have been the subject of sustained analysis in recent years. Such equations arise in models with rough forcing or highly singular interactions.
When the drift belongs only to a negative-order Besov or fractional Sobolev space, classical results for strong existence, uniqueness, and numerical approximation break down, and the interpretation of the equation
\begin{equation}\label{SDE:original}
    X_t = X_0 + \int_0^t b(s,X_s)\,ds + W_t
\end{equation}
requires non-standard analytical tools.

For function-valued drifts of low regularity and Brownian noise $W$, Zvonkin \cite{zvonkin_transformation_1974} and Veretennikov \cite{veretennikov_strong_1981} showed that a suitable solution to a specific PDE removes the drift, a technique nowadays referred to as Zvonkin transformation.
The first works applying this idea to time-homogeneous distributional drifts and Brownian noise are \cite{flandoli_sdes_2003,flandoli_sdes_2004}, in parallel to \cite{bass_stochastic_2001} where a different class of distributional drifts was treated using Dirichlet processes. More recently in \cite{flandoli_multidimensional_2017,delarue_rough_2016}, the time-inhomogeneous case was considered and subsequently extended to $\alpha$-stable noise in \cite{raynal_multidimensional_2022}. In the case of non-Markovian noise, such as fractional Brownian motion, PDE techniques cannot be applied. The first work in this setting is \cite{catellier_averaging_2016}, which develops an averaging operator along fractional Brownian trajectories to obtain pathwise well-posedness and regularisation by fractional noise for ODEs with distributional drift. This has been refined in subsequent works such as \cite{anzeletti_regularisation_2023}, which is based on limiting arguments and stochastic sewing techniques, first introduced in \cite{le_stochastic_2020}.
 There are other works complementing these results, such as \cite{cannizzaro_multidimensional_2018,perkowski_quantitative_2023,fitoussi_heat_2024}, as well as results in the $L^q_tL^p_x$ setting, for example \cite{krylov_strong_2005,zhang_strong_2005,zhang_stochastic_2010,beck_stochastic_2019,wei_stochastic_2023,wei_stochastic_2025}.

 The well-posedness of the equation studied in the present paper was developed in \cite{issoglio_pde_2024,issoglio_stochastic_2024}, where the authors refined the concept of \emph{virtual solutions} originally established in \cite{flandoli_multidimensional_2017}, for SDEs with drift in $C_T\mathcal{C}^{(-\beta)+}$, $\beta\in(0,\tfrac12)$ and Brownian noise. The precise definition of the Besov spaces $\mathcal{C}^{(-\beta)+}$ is given in Section \ref{sub:notationspacesnorms} below.
The idea is to apply a Zvonkin-type transformation, namely solve a Kolmogorov-type PDE to construct a function $\phi$ that removes the singular drift. Afterwards, solve a classical SDE with regular coefficients for the transformed process $Y_t = \phi(t,X_t)$, and then define the virtual solution as $X_t = \psi(t,Y_t)$, where $\psi$ is the inverse of $\phi$. In this way one obtains a canonical notion of solution even though the original drift is only a distribution and the term $\int_0^t b(s,X_s)\,ds$ is not directly well-defined.
\cite{issoglio_stochastic_2024} showed that \eqref{SDE:original} admits a unique virtual solution in any dimension, and it was later proved in \cite{chaparro_jaquez_convergence_2025} that this solution coincides with a strong probabilistic solution in the special case of dimension 1.

On the numerical side, for classical Lipschitz drifts, strong convergence of Euler-type schemes is well understood; cf.\ \cite{kloeden_numerical_1992}. There is a rich literature investigating convergence when the drift is irregular but still a function and the noise is Brownian, see for example \cite{yan_euler_2002,wang_euler-maruyama_2023,ngo_strong_2016,neuenkirch_adaptive_2019,neuenkirch_eulermaruyama_2021,gottlich_euler_2019}, and \cite{szolgyenyi_stochastic_2021} for a review.

The first work investigating the convergence on Euler-type schemes for one-dimensional SDEs with distributional drift and Brownian noise is \cite{de_angelis_numerical_2022} working in fractional Sobolev spaces, followed by \cite{chaparro_jaquez_convergence_2025} in Besov spaces of negative order, obtaining rates for the $L^1$ error. The details of their proofs are expanded below since their $L^1$ result will be used in the proofs of this work.

Recently, stochastic sewing-based techniques have allowed to obtain improved rates with Hölder continuous drifts and non-Markovian noise in the time-homogeneous case \cite{butkovsky_approximation_2021}, then extended to the time-inhomogeneous case \cite{bao_randomised_2025}, and in the case of Brownian noise with functional, but only integrable drift \cite{le_taming_2025}. In the case of fractional Brownian noise, a tamed Euler scheme for SDEs driven by fBm with drift in Besov spaces (including distributional drifts) is analysed in \cite{goudenege_numerical_2025}, where strong convergence rates depending on both the drift regularity and the Hurst parameter are obtained.

A strategy that has proven effective in the distributional drift setting, for both Brownian and fractional noise, is to replace the drift $b$ with a mollified drift $b^m$ 
 and then apply an Euler scheme to the mollified solutions. In the Brownian case considered here this corresponds to the classical Euler–Maruyama scheme. More precisely, the mollified SDE is
\begin{equation}
    \label{SDE:mollified}
    X^m_t = X_0 + \int_0^t b^m(s, X_s^m)\,ds + W_t,
\end{equation}
where $m\in\mathbb{N}$ is the mollification parameter, and the corresponding numerical discretisation is
\begin{equation}
    \label{SDE:mollifieddiscretised}
    X^{m,n}_t = X_0 + \int_0^t b^m(t_{k(s)}, X^{m,n}_{t_{k(s)}})\,ds + W_t,
\end{equation}
where $n\in\mathbb{N}$ is the number of time steps, $h := T/n$, $t_k := kh$ for $k=0,\dots,n$, and
\[
    k(s) := \max\{k\in\{0,\dots,n\}: t_k\le s\}.
\]

The aim of \cite{de_angelis_numerical_2022,chaparro_jaquez_convergence_2025} is to prove a bound of the form $\sup_{t\in[0,T]}\mathbb{E}|X_t - X^{m(n),n}_t|\le n^{-r}$, for some suitable rate $r$.  Explicit bounds are obtained after balancing the mollification parameter $m$ as a function of the numerical scheme step size $n$ for maximum convergence obtaining the following rate:
\begin{equation}
        \label{eq:rate}
        r=r(\beta,\varepsilon) := \frac{(\frac{1}{2} - \beta -\varepsilon)^2}{1+\beta + \varepsilon+ 2(\frac{1}{2} - \beta -\varepsilon)^2}.
\end{equation}
This rate is proved by controlling the difference between the transformed processes $Y$ and $Y^{m,n}$ arising from the virtual solution framework that yields well-posedness. Both the definition of virtual solution and the dynamics of the transformed process $Y$ are recalled in Section \ref{virtualsols}.

In those works the limitation to the $L^1$ error arises from the fact that the diffusion coefficient of the transformed process $Y$ possesses only spatial H\"older continuity, this prevents the direct use of It\^o or Tanaka formulas for the absolute value of the error in stronger $L^p$ norms.

To overcome the problems arising from low regularity of the diffusion coefficient and to study uniqueness of solutions to SDEs, Yamada and Watanabe \cite{watanabe_uniqueness_1971} introduced smooth absolute-value approximations (recalled in Section~\ref{subsec:YW-technique}).
Other authors have adapted this idea to study the convergence of numerical schemes with Hölder diffusion coefficient allowing It\^o calculus to be applied to the study of the $L^p$ error, for $p\geq1$; cf.\ \cite{gyongy_note_2011,ngo_strong_2016}.
These approximations produce error equations involving integral terms of the form
\[
\int_0^t |Y_s - Y^{m,n}_s|^\eta\,ds,
\]
for suitable exponents $\eta>1$, which can be handled using a Gronwall-type lemma from \cite{gyongy_note_2011}, restated below as Lemma~\ref{GR_Lemma32}.

The present article adapts this methodology to the one-dimensional Besov-drift
setting of \cite{chaparro_jaquez_convergence_2025}, yielding new $L^p$ convergence
results for the Euler–Maruyama approximation of \eqref{SDE:original}. 

Throughout
the paper we work under the following standing assumption on the drift.

\begin{assumption}{1}{A}
    \label{hypothesis:b}
    The drift satisfies
    \[
        b \in C_T^{\frac{1}{2}} \mathcal{C}^{(-\beta)+}
    \]
    for some $\beta \in (0,\frac{1}{2})$.
\end{assumption}
Under this assumption, in Theorem \ref{rate:Lp} we prove that for every $p\ge 2$ and every
$\varepsilon\in\bigl(0,\tfrac12-\beta\bigr)$ the scheme converges in the
$L^p$–sup norm.
\begin{theorem}[$L^p$-$\sup$ Convergence]
        \label{rate:Lp}
    Let Assumption \ref{hypothesis:b} hold. Then, for any choice of $p\geq 2$ and $\varepsilon\in(0,\frac12 - \beta)$ there exists a constant $C>0$ such that
    \begin{equation}
        \mathbb{E}\left[\sup_{t\in[0,T]} |X_t - X^{m(n),n}_t|^p\right]^{\frac{1}{p}}\leq C n^{-\frac{1}{p}r(\beta,\varepsilon)},
    \end{equation}
    where $r(\beta,\varepsilon)$ was defined in \eqref{eq:rate}.
\end{theorem}
\begin{remark}
    Note that this convergence rate is not uniform in $p$, as in \cite{gyongy_note_2011,ngo_strong_2016}, but the rate $r(\beta,\varepsilon)$ is exactly the one in \cite{chaparro_jaquez_convergence_2025}.
\end{remark}

Additionally, we prove the following result which extends the bound of \cite{chaparro_jaquez_convergence_2025} to the $L^1$-$\sup$ norm.
\begin{theorem}[$L^1$-$\sup$ Convergence]
        \label{rate:L1}
    Let Assumption \ref{hypothesis:b} hold. Then, there exists a constant $C>0$ such that for any $\varepsilon\in(0,\frac12 - \beta)$ it holds that
    \begin{equation}
        \mathbb{E}\left[\sup_{t\in[0,T]}\bigl|X_t-X_t^{m(n),n}\bigr|\right] \leq C\, n^{-r(\beta,\varepsilon)\left(\frac{1}{2} - \beta -\varepsilon\right)}, 
    \end{equation}
    where $r(\beta,\varepsilon)$ was defined in \eqref{eq:rate}.
\end{theorem}
\begin{remark}
    Throughout the rest of the paper we fix an auxiliary exponent
\[
   \hat{\beta}\in(\beta,\frac{1}{2})
\]
such that $b\in C^{1/2}_T\mathcal{C}^{-\hat{\beta}}$. The rates in Theorems~\ref{rate:Lp} and~\ref{rate:L1} depend only on $\beta$ from Assumption~\ref{hypothesis:b}, while the particular choice of $\hat{\beta}\in(\beta,\tfrac12)$ is only important when choosing optimal values for $m$ and $n$ to maximise convergence.
In the following $C$ will be a positive constant that may change from line to line.
\end{remark}
\noindent\textbf{Organisation of the paper.} Section~\ref{setting} recalls the analytic framework, the definition of virtual solutions, and other auxiliary bounds. Section \ref{sec:lp} contains the proof of Theorem~\ref{rate:Lp}, while Section \ref{sec:localtime} briefly gives an alternative method for proving the same result. Section \ref{sec:l1} gives the proof of Theorem~\ref{rate:L1}.
\section{Framework and Preliminaries}
\label{setting}
In this section we set up the framework in which we will work, including the formal definition of solution to the singular SDE. We also recall estimates on the numerical scheme coming from the literature and the construction of the Yamada-Watanabe approximation of the absolute value.
\subsection{Notation, Spaces, and Other Useful Estimates}\label{sub:notationspacesnorms}
Given a function of only one variable $g:\mathbb{R}\rightarrow\mathbb{R}$, we call its derivative $g'$.
Given a smooth enough function $f:[0,T]\times\mathbb{R}\rightarrow\mathbb{R}$, we will call $f_t$ and $f_x$ its partial derivatives with respect to time and space.
We write the quadratic variation of a continuous, real-valued stochastic process $(Z_t)_{t}$ as $\langle Z\rangle_t$.
Finally, we will contract expressions such as 
$C^\frac12([0,T]; L^\infty(\mathbb{R}))$ and $L^\infty([0,T]; C^1(\mathbb{R}))$
as $C_T^\frac12 L^\infty$ and $L^\infty_T C^1$, respectively.

We now give the precise definition of the H\"older-Zygmund spaces. We denote with $(\cdot)^\wedge$ and $(\cdot)^\vee$ the Fourier transform and its inverse over the space of
Schwartz functions $\mathcal{S}(\mathbb{R})$, extended in the usual way to $\mathcal{S}'(\mathbb{R})$.
For $\gamma\in\mathbb{R}$, $\mathcal{C}^\gamma$ is defined as
\begin{equation}
    \label{SPACE:Besov}
    \mathcal{C}^\gamma(\mathbb{R}) := \left\{ f\in\mathcal{S}'(\mathbb{R}): \|f\|_{\gamma} := \sup_{j\in\mathbb{N}\cup \{-1\}}  2^{j\gamma}\bigl\|(\varphi_j \hat{f})^{\vee}\bigr\|_{L^\infty}<+\infty  \right\}.
\end{equation}
Where $(\varphi_j)_j$ is a generic partition of the unity.
This space corresponds to the non-homogeneous Besov space $B^\gamma_{\infty \infty}(\mathbb{R},\mathbb{R})$. When $\gamma\in\mathbb{R}^+\setminus\mathbb{N}$, this space is exactly the classical H\"older space $C^\gamma$. For more details on the definition of these spaces and their relationships with other spaces see \cite{bahouri_fourier_2011} or \cite{sawano_theory_2018}.\\
For later use, we record equivalent norms in these spaces when $\gamma\in(0,1)$ and $\gamma\in(1,2)$:
\begin{align}
    \gamma\in(0,1)&:\  \|f\|_{L^\infty} + \sup_{0<|x-y|<1}\frac{|f(x) - f(y)|}{|x-y|^\gamma},\label{NORM:CGAMMA01} 
    \\
    \gamma\in(1,2)&:\  \|f\|_{L^\infty} +\|f'\|_{L^\infty} + \sup_{0<|x-y|<1}\frac{|f'(x) - f'(y)|}{|x-y|^{\gamma-1}}.\label{NORM:CGAMMA12} 
\end{align}
We also use the following notation for the supremum norm in time and space, and for a temporal Hölder seminorm for functions $g\in C_T\mathcal{C}^\kappa := C([0,T]; \mathcal{C}^{\kappa})$:
\begin{align*}
    \|g\|_{\infty,L^\infty} &:= \sup_{t\in[0,T]}\sup_{x\in\mathbb{R}}|g(t,x)|,\\
    [g]_{\kappa,L^\infty}   &:= \sup_{t,s\in[0,T],\ t\neq s}\frac{\|g(t) - g(s)\|_{L^\infty}}{|t-s|^\kappa}.
\end{align*}
Given $\gamma\in\mathbb{R}$ we denote as $\mathcal{C}^{\gamma+}$ the inductive spaces given by
\begin{equation*}
    \mathcal{C}^{\gamma+} := \bigcup_{\kappa>\gamma}\mathcal{C}^{\kappa}.
\end{equation*}
As for the non-inductive case, the notation $C_T\mathcal{C}^{\gamma+} := C([0,T]; \mathcal{C}^{\gamma+})$ will also be used. Since these spaces do not have a norm, by $f\in\mathcal{C}^{\gamma+}$ it is meant that there exist some $\kappa>\gamma$ such that $f\in\mathcal{C}^{\gamma}$.

Finally, we record two standard estimates in Hölder-Zygmund spaces that
will be used repeatedly in the sequel.

\begin{lemma}[Bernstein inequality]
    \label{Bernstein}
    For any $\gamma\in\mathbb{R}$ there exists $C>0$ such that for any $f\in\mathcal{C}^{\gamma + 1}$
    \[
        \|f'\|_\gamma \leq C\|f\|_{\gamma+1}, \qquad f\in\mathcal{C}^{\gamma+1}.
    \]
\end{lemma}
Let $(P_t)_{t\ge0}$ denote the heat semigroup on $\mathbb{R}$.
\begin{lemma}[Schauder estimates]
    \label{Schauder}
    Let $f\in\mathcal{C}^\gamma$ for some $\gamma\in\mathbb{R}$, and let $\theta\ge0$.
    Then there exists $C>0$ such that
    \[
        \|P_t f\|_{\gamma + 2\theta} \leq C t^{-\theta} \|f\|_\gamma,
        \qquad t\in(0,T].
    \]
    Moreover, if $\gamma\in\mathbb{R}$ and $\theta\in(0,1)$, then for all
    $f\in\mathcal{C}^{\gamma + 2\theta}$ there exists $C>0$ such that
    \[
        \|P_t f - f\|_{\gamma}\leq C t^{\theta} \|f\|_{\gamma + 2\theta},
        \qquad t\in(0,T].
    \]
\end{lemma}
\subsection{Virtual solutions}\label{virtualsols}

We recall the construction of virtual solutions from \cite{issoglio_pde_2024,issoglio_stochastic_2024} restricted in dimension 1 and the resulting well-posedness of \eqref{SDE:original} under the assumption $b\in C_T\mathcal{C}^{-\beta}$ for some $\beta\in(0,\tfrac12)$. 

Fix $\lambda>0$ and consider the backward Kolmogorov-type PDE
\begin{equation*}
  \begin{cases}
    u_t + \dfrac{1}{2}u_{xx} + b\,u_x = \lambda u - b,\\[0.2em]
    u(T,\cdot) = 0.
  \end{cases}
\end{equation*}
A solution is understood in the mild sense, i.e.\ as a function $u:[0,T]\times\mathbb{R}\to\mathbb{R}$ satisfying
\begin{equation}\label{eq:pde-singular}
  u(t)
  = \displaystyle\int_t^T P_{s-t}\bigl(u_x(s)b(s)\bigr)\,ds
    - \int_t^T P_{s-t}\bigl(\lambda u(s)-b(s)\bigr)\,ds,
  \qquad t\in[0,T].
\end{equation}
By \cite[Theorem~4.7]{issoglio_pde_2024} there exists a mild solution $u\in C_T\mathcal{C}^{(2-\beta)-}$, which is unique in $C_T\mathcal{C}^{(1+\beta)+}$. In particular, $u_x\in C_T\mathcal{C}^{(1-\beta)-}$, so $u_x$ is $\alpha$–Hölder continuous in space for every $\alpha<1-\beta$.

Moreover, \cite[Proposition~4.13]{issoglio_pde_2024} shows that, for $\lambda$ large enough, one has
\[
  \|u_x\|_{\infty,L^\infty} < \tfrac12.
\]
For such a choice of $\lambda$, define the Zvonkin-type transformation
\begin{equation*}
  \phi(t,x) := x + u(t,x), \qquad (t,x)\in[0,T]\times\mathbb{R}.
\end{equation*}
Then, for each $t\in[0,T]$, the map $x\mapsto\phi(t,x)$ is a bi-Lipschitz homeomorphism of $\mathbb{R}$, with Lipschitz constants bounded uniformly in $t$. We denote by
\[
  \psi(t,\cdot) := \phi(t,\cdot)^{-1}
\]
its inverse, which is also globally Lipschitz in $x$, uniformly in $t$.

Given $\phi$ and $\psi$, let us consider the SDE for the transformed process $Y$(formally $Y_t = \phi(t, X_t)$):
\begin{equation}\label{eq:Ydyn}
  Y_t
  = \phi(0,X_0)
    + \lambda\int_0^t u\bigl(s,\psi(s,Y_s)\bigr)\,ds
    + \int_0^t \bigl(1 + u_x(s,\psi(s,Y_s))\bigr)\,dW_s,
  \qquad t\in[0,T].
\end{equation}
The coefficients of this SDE are bounded and globally Lipschitz in the space variable, so the SDE admits a unique strong solution $Y$. Following \cite{flandoli_multidimensional_2017}, one then defines
\[
  X_t := \psi(t,Y_t), \qquad t\in[0,T],
\]
and calls $X$ a \emph{virtual solution} of \eqref{SDE:original}. It is shown in \cite[Theorem~4.5]{issoglio_stochastic_2024} that under $b\in C_T\mathcal{C}^{(-\beta)+}$ with $\beta\in(0,\tfrac12)$ there exists a unique virtual solution to \eqref{SDE:original}.

For later use we also introduce the corresponding objects associated with a mollified drift $b^m$. Let $u^m$ be the solution to the Kolmogorov equation with $b^m$ in place of $b$, and define
\[
  \phi^m(t,x) := x + u^m(t,x), \qquad
  \psi^m(t,\cdot) := (\phi^m(t,\cdot))^{-1}.
\]
By the same arguments as above (see again \cite[Proposition~4.13]{issoglio_pde_2024}) the maps $x\mapsto\phi^m(t,x)$ and $x\mapsto\psi^m(t,x)$ are bi-Lipschitz homeomorphisms of $\mathbb{R}$, with Lipschitz constants bounded uniformly in both $t$ and $m$. We denote by $Y^m$ and $X^m_t := \psi^m(t,Y^m_t)$ the transformed and virtual solutions associated to the drift $b^m$, with dynamics given by:
\begin{equation}\label{eq:Ymdyn}
  Y^m_t
  = \phi^m(0,X_0)
    + \lambda\int_0^t u^m\bigl(s,\psi^m(s,Y^m_s)\bigr)\,ds
    + \int_0^t \bigl(1 + u_x^m(s,\psi^m(s,Y^m_s))\bigr)\,dW_s,
  \qquad t\in[0,T].
\end{equation}
In dimension $1$, virtual solutions are strong probabilistic solutions; see \cite[Lemma~2.3]{chaparro_jaquez_convergence_2025}, and if $b$ is a bounded and measurable function then virtual solutions coincide with strong solutions.

\subsection{Auxiliary Estimates}
\subsubsection*{A Gronwall-type lemma for integral error terms}
The following inequality, due to Gyöngy and Rásonyi \cite{gyongy_note_2011},
is a stochastic Gronwall–type lemma that controls the supremum of a
non-negative process in terms of suitable integral functionals of the same
process. It will be applied in the later sections to the
Yamada–Watanabe–type approximations of the error processes.
\begin{lemma}[\text{\cite[Lemma~3.2]{gyongy_note_2011}}]
    \label{GR_Lemma32}
    Let $(Z_t)_{t\geq0}$ be a non-negative stochastic process and set
    $V_t=\sup_{s\leq t}Z_s$. Assume that for some $p>0$, $q\geq 1$,
    $\varrho\in [1,q]$, and constants $C,R\geq 0$ one has
    \[
        \mathbb{E}\left[ V_t^p \right] \leq
        C \mathbb{E}\left[ \left( \int_0^t V_s \, ds \right)^p \right]
        + C \mathbb{E}\left[ \left( \int_0^t Z_s^\varrho \, ds \right)^{p/q} \right]
        + R < +\infty,
    \]
    for all $t\geq 0$. Then, for each $T\geq 0$, if either $p\geq q$ or both
    $\varrho< q$ and $p>q+1-\varrho$ hold, there exist constants $C_1,C_2>0$ depending
    on $C$, $T$, $\varrho$, and $p$ such that
    \[
        \mathbb{E}\left[ V_T^p \right] \leq
        C_1 R + C_2 \int_0^T \mathbb{E}\left[ Z_s \right] ds.
    \]
\end{lemma}
As in the H\"older continuous case of \cite{gyongy_note_2011} and \cite{ngo_strong_2016} this inequality will be used to bound the $L^p$ error of our scheme with the $1/p$-root of the $L^1$ error estimates already available.
\subsubsection*{Error bounds for the mollified Euler scheme}
While bounding the numerical error in the proofs of the main theorems we will need the dynamics of the Zvonkin transform of the Euler scheme for $X^m$. By \cite[Lemma~6.1]{de_angelis_numerical_2022} we have,
\begin{align}
    \label{eq:Ymn}
     Y^{m,n}_t
  = \phi^m(0,X_0)
    + \lambda\int_0^t u^m\bigl(s,\psi^m(s,Y^{m,n}_s)\bigr)\,ds
    + \int_0^t \bigl(1 + u_x^m(s,\psi^m(s,Y^{m,n}_s))\bigr)\,dW_s\\
    + \int_0^t \bigl(b^{m}\bigl(t_{k(s)},X_{t_{k(s)}}^{m,n}\bigr)-b^{m}\bigl(s,X_s^{m,n}\bigr)\bigr)\bigl(1+u^{m}_x(s,X_s^{m,n})\bigr)ds
  ,\qquad t\in[0,T].
\end{align}
Next, we recall the $L^1$–error estimate for the Euler–Maruyama approximation of the mollified SDE \eqref{SDE:mollified}. This estimate will be used in both in the
$L^1$ and $L^p$ convergence analyses when we separate the “numerical” error
$X^{m,n}-X^m$ from the “stability” error $X^m-X$.

\begin{lemma}[\text{\cite[Lemma~6.2]{de_angelis_numerical_2022}}]
    \label{DGI:Lemma62}
    Assume $b^m\in C^\frac{1}{2}_T L^\infty\cap L^\infty_T C^1_b$. Then 
    \begin{equation}
        \label{numerica_L1}
        \sup_{0\leq t \leq T}\mathbb{E}|X^{m,n}_t - X^m_t|
        \leq C (A_m\, n^{-1} + B_m\, n^{-\frac{1}{2}}),
    \end{equation}
    where
    \begin{align*}
        A_m &= \|b^m\|_{\infty, L^\infty}
              \bigl(1+\|b^m_x\|_{\infty,L^\infty}\bigr), \\
        B_m &= \|b_x^m\|_{\infty,L^\infty}
              + [b^m]_{\frac{1}{2},L^\infty}.
    \end{align*}
\end{lemma}

\subsubsection*{Stability estimates for the Zvonkin transformation}

Now, we collect the stability estimates for the Zvonkin transformation and its
inverse, as well as the corresponding stability bound for the virtual
solutions. These results allow us to control the error between $X$ and $X^m$
via the transformed processes $Y$ and $Y^m$ in the proofs of the main theorems.

\begin{lemma}[\text{\cite[Lemma~4.3]{chaparro_jaquez_convergence_2025}}]
    \label{jaquez:L13}
    Let $\hat\beta\in(\beta,\frac{1}{2})$, then
    \begin{equation*}
        \sup_{(t,y)\in[0,T]\times\mathbb{R}}
        |\psi(t,y) - \psi^m(t,y)|
        \leq 2C\|b-b^m\|_{C_T\mathcal{C}^{-\hat\beta}}.
    \end{equation*}
    Where the constant C depends on $T$ and $\|b\|_{C_T\mathcal{C}^{-\hat\beta}}$.
\end{lemma} 

\begin{lemma}[\text{\cite[Lemma~4.4]{chaparro_jaquez_convergence_2025}}]
    \label{jaquez:L14}
    Let $\hat\beta\in(\beta,\frac{1}{2})$. Then, for any $\alpha<1-\hat\beta$, and any $y,y'\in\mathbb{R}$, the following bounds hold:
    \begin{equation}\label{jaquez:L141}
        |u^m(s,\psi^m(s,y'))-u(s,\psi(s,y))|
        \leq 2C\|b^m-b\|_{C_T\mathcal{C}^{-\hat\beta}} + |y-y'|,
    \end{equation}
    \begin{align}\label{jaquez:L142}
        &|u^m_x(s,\psi^m(s,y'))-u_x(s,\psi(s,y))|
        \leq C\|b^m-b\|_{C_T\mathcal{C}^{-\hat\beta}} \nonumber\\
        &\qquad+ 2^\alpha C^\alpha\|u\|_{C_T\mathcal{C}^{1+\alpha}}
                 \|b^m-b\|_{C_T\mathcal{C}^{-\hat\beta}}^\alpha
        + 2||u_x||_{C_T\mathcal{C}^{\alpha}}|y'-y|^\alpha.
    \end{align}
    Where $C$ is the same one as Lemma~\ref{jaquez:L13}.
\end{lemma}
\begin{lemma}[\text{\cite[Proposition 3.3]{chaparro_jaquez_convergence_2025}}]
    \label{jaquez:P6}
    Under Assumption \ref{hypothesis:b}, let $\hat{\beta}\in(\beta,\frac{1}{2})$ be such that $b\in C^{1/2}_T\mathcal{C}^{-\hat{\beta}}$. Then for any $\alpha\in(\frac{1}{2},1-\hat{\beta})$ there exists a constant $C>0$ such that
    \begin{equation*}
        \sup_{t\in[0,T]} \mathbb{E}|X^m_t - X_t|
        \leq C \|b^m - b\|_{C_T\mathcal{C}^{-\hat{\beta}}}^{2\alpha - 1},
    \end{equation*}
    for all $m\in\mathbb{N}^+$ such that
    $\|b^m - b\|_{C_T\mathcal{C}^{-\hat{\beta}}}<1$.
\end{lemma}
\subsubsection*{The Approximation Technique of Yamada and Watanabe}\label{subsec:YW-technique}
One of the first steps in the proofs of Theorems \ref{rate:L1} and \ref{rate:Lp} is to approximate the function $x\mapsto |x|$
in the way of Yamada and Watanabe \cite{watanabe_uniqueness_1971}.
 Let $\delta>1$ and $\kappa\in(0,1)$, and define first a function $\psi_{\delta,\kappa}:\mathbb{R}\rightarrow[0,\infty)$ 
such that:
\begin{enumerate}
    \item $ \supp\psi_{\delta,\kappa} \subset [\kappa/\delta, \kappa],$
    \item $\displaystyle \int_{\kappa/\delta}^{\kappa}\psi_{\delta,\kappa}(z)\,dz =1,$
    \item $0\leq \psi_{\delta,\kappa}(z) \leq \dfrac{2}{z\log\delta},\ z>0.$
\end{enumerate}
Such a function $\psi_{\delta,\kappa}$ exists: for example one may normalise 
the function $z \mapsto 1/(z\log\delta)$ on $[\kappa/\delta,\kappa]$.
We now define another function 
$\varphi_{\delta,\kappa}\in C^2(\mathbb{R};\mathbb{R})$ as
\[
  \varphi_{\delta,\kappa}(x):=\int_{0}^{|x|}\int_0^y\psi_{\delta,\kappa}(z)\,dz\,dy.
\]
$\varphi_{\delta,\kappa}$ has the following properties:
\begin{align}
    & |x|\leq \kappa + \varphi_{\delta,\kappa}(x)\text{, for any } x\in\mathbb{R}, \label{propphi:A}\\
    & 0\leq |\varphi'_{\delta,\kappa}(x)|\leq 1\text{, for any } x\in\mathbb{R},  \label{propphi:B}\\
    & \varphi''_{\delta,\kappa}(\pm |x|) = \psi_{\delta,\kappa}(|x|)\leq \frac{2}{|x|\log\delta}\mathbbm{1}_{[\kappa/\delta,\kappa]}(|x|)\text{, for any } x\in\mathbb{R}\setminus\{0\}.  \label{propphi:C}
\end{align}
Property \eqref{propphi:A} will allow us to bound the absolute value in the following, and since $\varphi_{\delta,\kappa}\in C^2$, we will be able to apply the classical It\^o formula.
\section{Proof of Theorem~\ref{rate:Lp}: $L^p$-sup Convergence}
In this section the main convergence rate in $L^p$ is proved using the Yamada-Watanabe approximation technique. In the last subsection, we suggest an alternative proof of the same result.
\subsection{A Yamada-Watanabe Approximation Based Proof}
\label{sec:lp}
We begin with the standard decomposition of the strong error into a
``stability'' part and a ``numerical'' part. For any $p\ge2$ we have
\begin{align}\label{eq:decomposition}
    \mathbb{E}\left[ \sup_{0\leq t\leq T} \left\lvert X_t^{m,n}-X_t \right\rvert^p \right]
    &\leq \mathbb{E}\left[ \sup_{0\leq t\leq T}\left(\left\lvert X_t^{m,n}-X_t^{m}\right\rvert+\left\lvert X_t^{m}-X_t\right\rvert\right)^p\right]\nonumber \\
    &\leq 2^{p-1}\mathbb{E}\left[ \sup_{0\leq t\leq T}\left\lvert X_t^{m,n}-X_t^{m}\right\rvert^p\right]
    +2^{p-1}\mathbb{E}\left[ \sup_{0\leq t\leq T}\left\lvert X_t^{m}-X_t\right\rvert^p\right].
\end{align}

Our goal is to obtain bounds for the two terms in the second inequality in the form required by Lemma~\ref{GR_Lemma32}.

\begin{lemma}[Moment Stability Estimate]
    \label{Lp:stability}
    Let Assumption \ref{hypothesis:b} hold and fix $\hat{\beta}\in(\beta,\tfrac12)$ as above. Let $\alpha\in(\tfrac12,1-\hat{\beta})$ and $p\ge2$. Then, for all $m\in\mathbb{N}^+$ such that
    $\|b^m - b\|_{C_T\mathcal{C}^{-\hat{\beta}}}<1$ there exists a constant $C>0$ such that
    \begin{equation*}
        \mathbb{E}\left[ \sup_{0\leq t\leq T}\left\lvert X_t^{m}-X_t\right\rvert^p\right] \leq C\|b^m-b\|_{C_T\mathcal{C}^{-\hat{\beta}}}^{2\alpha-1}.
    \end{equation*}
\end{lemma}
\begin{proof}
Let $Y$ and $Y^m$ denote the transformed processes associated with $X$ and
$X^m$, respectively, as in Section~\ref{virtualsols}. By the bi-Lipschitz
property of $\psi$ and $\psi^m$ and Lemma~\ref{jaquez:L13} we have
\begin{align*}
    \left\lvert X_{t}^{m}-X_t\right\rvert^p
    &=\left\lvert \psi^{m}\bigl(t,Y^{m}_t\bigr)-\psi\bigl(t,Y_t\bigr)\right\rvert^p\\
    &\leq 2^{p-1}\left\lvert \psi^{m}\bigl(t,Y^{m}_t\bigr)-\psi^{m}\bigl(t,Y_t\bigr)\right\rvert^p
         +2^{p-1}\left\lvert \psi^{m}\bigl(t,Y_t\bigr)-\psi\bigl(t,Y_t\bigr)\right\rvert^p\\
    &\leq 2^{p-1}\left\lvert Y^{m}_t-Y_t\right\rvert^p
         +2^{p-1}C^p \left\lVert b^{m}-b\right\rVert_{C_T\mathcal{C}^{-\hat{\beta}}}^p.
\end{align*}
Taking expectation of the supremum, we have
\begin{equation}
\label{eq:stability-reduction-Y}
\mathbb{E}\left[ \sup_{0\leq t\leq T}\left\lvert X_t^{m}-X_t\right\rvert^p\right]
\leq2^{p-1}C^p \left\lVert b^{m}-b\right\rVert_{C_T\mathcal{C}^{-\hat{\beta}}}^p +
 2^{p-1}\mathbb{E}\left[ \sup_{0\leq t\leq T}\left\lvert Y^{m}_t-Y_t\right\rvert^p\right].
\end{equation}
It is thus enough to estimate the second term.

Using \eqref{propphi:A} and It\^o Formula on $\varphi_{\delta,\kappa}$, we have
\begin{align}
    &\mathbb{E}\left[ \sup_{0\leq t\leq T}\left\lvert Y^{m}_t-Y_t\right\rvert^p\right]
    \leq \mathbb{E}\left[ \sup_{0\leq t\leq T}\bigl(\kappa+\varphi_{\delta,\kappa}(Y^{m}_t-Y_t)\bigr)^p\right]\nonumber\leq4^{p-1}\kappa^p\\
    &\quad
    +4^{p-1}\mathbb{E}\left[ \sup_{0\leq t\leq T}\left(\int_{0}^{t}\varphi'_{\delta,\kappa}\bigl(Y^{m}_s-Y_s\bigr)\bigl(u^{m}(s,\psi^{m}(s,Y^{m}_s))-u\bigl(s,\psi(s,Y_s)\bigr)\bigr)\,ds\right)^p\right]\label{Lp:stability:line1}\\
    &\quad+4^{p-1}\mathbb{E}\left[ \sup_{0\leq t\leq T}\left(\int_{0}^{t}\varphi'_{\delta,\kappa}\bigl(Y^{m}_s-Y_s\bigr)\bigl(u^{m}_x(s,\psi^{m}(s,Y^{m}_s))-u_x(s,\psi(s,Y_s))\bigr)\,dW_s\right)^p\right]\label{Lp:stability:line2}\\
    &\quad+4^{p-1}\mathbb{E}\left[ \sup_{0\leq t\leq T}\left(\int_{0}^{t}\varphi''_{\delta,\kappa}\bigl(Y^{m}_s-Y_s\bigr)\bigl(u^{m}_x(s,\psi^{m}(s,Y^{m}_s))-u_x(s,\psi(s,Y_s))\bigr)^2\,ds\right)^p\right].\label{Lp:stability:line3}
\end{align}
By Lemma~\ref{jaquez:L14} and $|\varphi'_{\delta,\kappa}|\le1$,
\begin{align*}
    \eqref{Lp:stability:line1}
    &\leq4^{p-1}\mathbb{E}\left[\left(\int_{0}^{T}\bigl(2c\left\lVert b^{m}-b\right\rVert_{C_T\mathcal{C}^{-\hat{\beta}}}+\left\lvert Y^{m}_s-Y_s\right\rvert\bigr)\,ds\right)^p\right]\\
    &\leq C\left\lVert b^{m}-b\right\rVert_{C_T\mathcal{C}^{-\hat{\beta}}}^p
    +C\mathbb{E}\left[\left(\int_{0}^{T}\sup_{0\leq r\leq T}\left\lvert Y^{m}_r-Y_r\right\rvert\,ds\right)^p\right].
\end{align*}
For \eqref{Lp:stability:line3}, using \eqref{propphi:C} and Lemma~\ref{jaquez:L14},
\begin{align*}
    \eqref{Lp:stability:line3}
    &\leq C\mathbb{E}\left[\sup_{0\leq t\leq T}\left(\int_{0}^{t}\frac{2\mathbbm{1}_{\left[\kappa/\delta,\kappa\right]}(|Y^{m}_s-Y_s|)}{|Y^{m}_s-Y_s|\log\delta}\right.\right.\times\\
    &\quad\left.\left.\times\bigl(\|u_x\|_{C_T\mathcal{C}^\alpha}|Y^{m}_s-Y_s|^{\alpha}
    +C\|b^{m}-b\|_{C_T\mathcal{C}^{-\hat{\beta}}}
    +C\|b^{m}-b\|_{C_T\mathcal{C}^{-\hat{\beta}}}^\alpha\bigr)^2\,ds\right)^p\right]\\
    &\leq C\left(\frac{2}{\log\delta}\kappa^{2\alpha-1}+
    \frac{2}{\kappa\log\delta}\|b^{m}-b\|_{C_T\mathcal{C}^{-\hat{\beta}}}^2
    +\frac{2}{\kappa\log\delta}\|b^{m}-b\|_{C_T\mathcal{C}^{-\hat{\beta}}}^{2\alpha}
    \right)^p.
\end{align*}
For the martingale term \eqref{Lp:stability:line2}, the Burkholder-Davis-Gundy (BDG) inequality combined with Lemma~\ref{jaquez:L14} yields
\begin{align*}
    \eqref{Lp:stability:line2}
    &\leq C\mathbb{E}\left[\left(\int_{0}^{T}\bigl(u^{m}_x(s,\psi^{m}(s,Y^{m}_s))-u_x(s,\psi(s,Y_s))\bigr)^2\,ds\right)^{\frac{p}{2}}\right]\\
    &\leq C\left(\|b^{m}-b\|_{C_T\mathcal{C}^{-\hat{\beta}}}^p
    +\|b^{m}-b\|_{C_T\mathcal{C}^{-\hat{\beta}}}^{\alpha p}\right)
    +C\mathbb{E}\left[\left(\int_{0}^{T}\left\lvert Y^{m}_s-Y_s\right\rvert^{2\alpha}\,ds\right)^{\frac{p}{2}}\right].
\end{align*}
Collecting the bounds on \eqref{Lp:stability:line1},\eqref{Lp:stability:line2}, and \eqref{Lp:stability:line3} and choosing
$\kappa := \|b^m-b\|_{C_T\mathcal{C}^{-\hat{\beta}}}$, we see that
$\mathbb{E}\left[ \sup_{0\leq t\leq T}\left\lvert Y^{m}_t-Y_t\right\rvert^p\right]$ satisfies an inequality of the form required by
Lemma~\ref{GR_Lemma32}, with
\[
  Z_s = |Y^{m}_t-Y_t|,\qquad
  q=2,\qquad
  \varrho=2\alpha,\qquad
  p\ge2,
\]
the other terms are all powers greater than $2\alpha-1$ of $\|b^m-b\|_{C_T\mathcal{C}^{-\hat{\beta}}}$, in the notation of
Lemma~\ref{GR_Lemma32} we collect them into $R$. Since we can take $\alpha>\tfrac12$, the condition
$\varrho<q$ and $p>q+1-\varrho$ is satisfied.

Applying Lemma~\ref{GR_Lemma32} yields
\[
  \mathbb{E}[V_T^p]
  \leq
  C_1 R
  + C_2\int_0^T\mathbb{E}[Z_s]\,ds
  \leq
  C\|b^m-b\|_{C_T\mathcal{C}^{-\hat{\beta}}}^{2\alpha-1},
\]
where in the last step we use the $L^1$ stability estimate
Lemma~\ref{jaquez:P6} to bound $\mathbb{E}[Z_s]$ in terms of
$\|b^m-b\|_{C_T\mathcal{C}^{-\hat{\beta}}}^{2\alpha-1}$, and the fact that
$\|b^m-b\|_{C_T\mathcal{C}^{-\hat{\beta}}}<1$ to absorb the higher order terms. Together with \eqref{eq:stability-reduction-Y} this completes
the proof.
\end{proof}

Now we move on to the estimate of the numerical approximation of the solution to the mollified SDE.
\begin{lemma}[Moment Numerical Estimate]\label{lp:numerical}
    Let Assumption \ref{hypothesis:b} hold and $p\ge2$. Then it holds that
   \begin{equation*}
    \mathbb{E}\left[\sup_{0\leq t\leq T}\left\lvert X_t^{m,n}-X_t^{m}\right\rvert^p\right]
    \leq  C(A_m\, n^{-1} + B_m\, n^{-\frac{1}{2}}+C_m n^{-p} + D_m n^{-\frac{p}{2}}),
\end{equation*}
where $A_m$ and $B_m$ are the same ones as in Lemma~\ref{DGI:Lemma62} and $C_m$ and $D_m$ are defined as:
\begin{align*}
    C_m &:= \| b^{m} \|_{\infty,L^\infty}^p(\| b^{m}_x \|_{\infty,L^\infty}^p+1),\\
    D_m &:= \| b^{m}_x \|_{\infty,L^\infty}^p+[ b^{m} ]_{\frac{1}{2},L^\infty}^p.
\end{align*}
\end{lemma}
\begin{proof}
    The argument parallels that of the stability estimate. Using the Zvonkin transformation for $X^{m,n}$ and $X^m$, we first reduce to bounding $Y^{m,n}-Y^{m}$.
\begin{align}
    &\mathbb{E}\left[\sup_{0\leq t\leq T}\left\lvert X_t^{m,n}-X_t^{m}\right\rvert^p\right]
    \leq C\mathbb{E}\left[\sup_{0\leq t\leq T}\left\lvert Y^{m,n}_t-Y^{m}_t\right\rvert^p\right]\nonumber\\
    &\leq\mathbb{E}\left[\sup_{0\leq t\leq T}\left(\kappa+\varphi_{\delta,\kappa}\bigl(Y^{m,n}_t-Y^{m}_t\bigr)\right)^p\right] \leq C\kappa^p\nonumber\\
   & +C\mathbb{E}\left[\sup_{0\leq t\leq T}\left(\int_{0}^{t}\varphi'_{\delta,\kappa}\bigl(Y^{m,n}_s-Y^{m}_s\bigr)\bigl(u^{m}(s,\psi^{m}(s,Y^{m,n}_s))-u^{m}(s,\psi^{m}(s,Y^{m}_s))\bigr)\,ds\right)^p\right]\label{Lp:numerical:line1}\\
    &+C\mathbb{E}\left[\sup_{0\leq t\leq T}\left(\int_{0}^{t}\varphi'_{\delta,\kappa}\bigl(Y^{m,n}_s-Y^{m}_s\bigr)\bigl(b^{m}\bigl(t_{k(s)},X_{t_{k(s)}}^{m,n}\bigr)-b^{m}\bigl(s,X_s^{m,n}\bigr)\bigr)\bigl(1+u^{m}_x(s,X_s^{m,n})\bigr)\,ds\right)^p\right]\label{Lp:numerical:line2}\\
    &+C\mathbb{E}\left[\sup_{0\leq t\leq T}\left(\int_{0}^{t}\varphi'_{\delta,\kappa}\bigl(Y^{m,n}_s-Y^{m}_s\bigr)\bigl(u^{m}_x(s,\psi^{m}(s,Y^{m,n}_s))-u^{m}_x(s,\psi^{m}(s,Y^{m}_s))\bigr)\,dW_s\right)^p\right]\label{Lp:numerical:line3}\\
    &+C\mathbb{E}\left[\sup_{0\leq t\leq T}\left(\int_{0}^{t}\varphi''_{\delta,\kappa}\bigl(Y^{m,n}_s-Y^{m}_s\bigr)\bigl(u^{m}_x(s,\psi^{m}(s,Y^{m,n}_s))-u^{m}_x(s,\psi^{m}(s,Y^{m}_s))\bigr)^2\,ds\right)^p\right].\label{Lp:numerical:line4}
\end{align}

We first focus on
\eqref{Lp:numerical:line2}, which encodes the time discretisation error. Splitting \eqref{Lp:numerical:line2} into a ``time'' and a ``space'' term,
\begin{align}
    \eqref{Lp:numerical:line2}
    &\leq C\mathbb{E}\left[\left(\int_{0}^{T}\left\lvert b^{m}\bigl(t_{k(s)},X_{t_{k(s)}}^{m,n}\bigr)-b^{m}\bigl(s,X_{t_{k(s)}}^{m,n}\bigr)\right\rvert\,ds\right)^p\right]\label{Lp:numerical:line2:time}\\
    &\quad+C\mathbb{E}\left[\left(\int_{0}^{T}\left\lvert b^{m}\bigl(s,X_{t_{k(s)}}^{m,n}\bigr)-b^{m}\bigl(s,X_s^{m,n}\bigr)\right\rvert\,ds\right)^p\right].\label{Lp:numerical:line2:space}
\end{align}
For the time part \eqref{Lp:numerical:line2:time}, the temporal $1/2$-Hölder regularity of $b^m$ yields
\[
    \eqref{Lp:numerical:line2:time}
    \leq cT^{\frac{3}{2}p}[b^m]_{\frac{1}{2},L^\infty}^p n^{-\frac{p}{2}}.
\]
For the space part \eqref{Lp:numerical:line2:space}, recall that by the definition of the Euler scheme\eqref{SDE:mollifieddiscretised}, we have
\[
\mathbb{E}\left[\left\lvert X_t^{m,n}-X_{t_{k(t)}}^{m,n}\right\rvert^p\right]
    \leq 2^{p-1}\|b^{m}\|_{\infty,L^\infty}^p\frac{T^p}{n^p}
    +2^{p-1}\frac{T^{\frac{p}{2}}}{n^{\frac{p}{2}}}\,\mathbb{E}\bigl[|W_1|^p\bigr],
\]
which then implies the bound for \eqref{Lp:numerical:line2:space}, since
\begin{align*}
    \eqref{Lp:numerical:line2:space}\leq T^p\|b_x^m\|_{\infty,L^\infty}^p\mathbb{E}\left[\left\lvert X_t^{m,n}-X_{t_{k(t)}}^{m,n}\right\rvert^p\right]
\end{align*}
The terms \eqref{Lp:numerical:line1}, \eqref{Lp:numerical:line3} and
\eqref{Lp:numerical:line4} are handled similarly as in the proof of
Lemma~\ref{Lp:stability}, using the Lipschitz continuity and Hölder
regularity of $\psi^m$,$u^m$, and $u_x^m$. 
For \eqref{Lp:numerical:line1} we obtain,
\[
    \eqref{Lp:numerical:line1}
    \leq C\|u_x^m\|_{\infty,L^\infty}^p\mathbb{E}\left[\left(\int_{0}^{T}\sup_{0\leq r\leq T}\left\lvert Y^{m,n}_r-Y^{m}_r\right\rvert\,ds\right)^p\right].
\]
When dealing with \eqref{Lp:numerical:line3}, BDG and Hölder continuity of $u_x^m$ give
\[
    \eqref{Lp:numerical:line3}
    \leq C\mathbb{E}\left[\left(\int_{0}^{T}\left\lvert Y^{m,n}_s-Y^{m}_s\right\rvert^{2\alpha}\,ds\right)^{\frac{p}{2}}\right].
\]
The second derivative term \eqref{Lp:numerical:line4} is treated using \eqref{propphi:C}, yielding
\[
    \eqref{Lp:numerical:line4}\le C\frac{2^p}{(\log\delta)^p}T^p\kappa^{(2\alpha-1)p}.
\]

Collecting all bounds and applying Lemma~\ref{GR_Lemma32} with $q=2$ and $\varrho=2\alpha$ to $
Z_s:=|Y^{m,n}_s-Y^{m}_s|$ gives,

\begin{align*}
    \mathbb{E}\left[\sup_{0\leq t\leq T}\left\lvert Y_t^{m,n}-Y_t^{m}\right\rvert^p\right]\leq&\, \kappa^p + \kappa^{(2\alpha-1)p} +  \int_{0}^{T}\mathbb{E}|Y_s^{m,n}-Y_s^{m}|ds\\
    +& T^p\|b_x^m\|_{\infty,L^\infty}^p\Bigg(2^{p-1}\|b^{m}\|_{\infty,L^\infty}^p\frac{T^p}{n^p} +2^{p-1}\frac{T^{\frac{p}{2}}}{n^{\frac{p}{2}}}\,\mathbb{E}\bigl[|W_1|^p\bigr]\Bigg)
\end{align*}
Using Lemma~\ref{DGI:Lemma62} and letting $\kappa\to0$ gives the stated estimate.
\end{proof}
\begin{proof}[Proof of Theorem~\ref{rate:Lp}]
    With Lemmas \ref{Lp:stability} and \ref{lp:numerical} in hand, the optimisation of $m(n)=n^\gamma$ to maximise the convergence rate of the bound in $n$ of \eqref{eq:decomposition} proceeds exactly as in the proof of \cite[Theorem 3.4]{chaparro_jaquez_convergence_2025}, since the slowest term of the numerical error is still $\|b_x^m\|_{\infty,L^\infty}n^{-\frac{1}{2}}$. After choosing $\gamma$ for maximum convergence speed we get the same rate \eqref{eq:rate}. Taking the $p$-th root on both sides produces the additional factor $1/p$ in the exponent.
\end{proof}

\subsection{An Equivalent Local Time-Based Approach}\label{sec:localtime}
Theorem~\ref{rate:Lp} can also be proven by using a similar technique to \cite{chaparro_jaquez_convergence_2025,de_angelis_numerical_2022}. Here we give the broad strategy while the details are left to the reader.
In \cite{chaparro_jaquez_convergence_2025,de_angelis_numerical_2022}, the $L^1$ rates are obtained using the Ito-Tanaka formula and then bounding the local time term using the following inequality.
\begin{lemma}{\cite[Lemma~5.1]{de_angelis_numerical_2022}}\label{lem:localtimel1}
    For any $\varepsilon\in(0,1)$ and any real-valued, continuous semimartingale $\hat{Y}$ we have 
    \begin{align}
        \mathbb{E}\big[L^0_t(\hat{Y})\big]\leq 4\varepsilon - 2\mathbb{E}&\left[\int_0^t \left(\mathbbm{1}_{\{\hat{Y}\in(0,\varepsilon)\}} + \mathbbm{1}_{\{\hat{Y}_s\geq\varepsilon\}}e^{1-\hat{Y}_s/\varepsilon}\right)\mathrm{d}\hat{Y}_s\right]\nonumber\\
        +\frac{1}{\varepsilon}\mathbb{E}&\left[\int_0^t\mathbbm{1}_{\{\hat{Y}>\varepsilon\}} e^{\{1-\hat{Y}_s/\varepsilon\}}\mathrm{d}\langle\hat{Y}\rangle_s\right]\label{localtimeformula}.
    \end{align}
\end{lemma}
 instead of approximating the absolute value using the smooth sequence defined in Subsection \ref{subsec:YW-technique}, an $L^p$ analogue of Lemma~\ref{lem:localtimel1} is needed. 
\begin{lemma}\label{localtimelp}
     For any $\varepsilon\in(0,1)$, $p\ge2$, and any real-valued, continuous semimartingale $\hat{Y}$ we have
     \begin{align*}
        \mathbb{E}\left[\sup_{0\leq t\leq T}(L^0_t(\hat{Y}))^p\right]&=\mathbb{E}\left[(L^0_T(\hat{Y}))^p\right]\\
        &\le 4^p3^{p-1}\varepsilon^p + 3^{p-1}2^p\mathbb{E}\left[\left(\int_0^T \left(\mathbbm{1}_{\{\hat{Y}\in(0,\varepsilon)\}} + \mathbbm{1}_{\{\hat{Y}_s\geq\varepsilon\}}e^{1-\hat{Y}_s/\varepsilon}\right)\mathrm{d}\hat{Y}_s \right)^p\right]\\
        &+\frac{3^{p-1}}{\varepsilon^{p}}\mathbb{E}\left[\left(\int_0^t\mathbbm{1}_{\{\hat{Y}>\varepsilon\}} e^{\{1-\hat{Y}_s/\varepsilon\}}\mathrm{d}\langle\hat{Y}\rangle_s\right)^p\right].
     \end{align*}
\end{lemma}
This result can be proven by raising \eqref{localtimeformula} to the power of $p$ and using that $(\sum_{i=1}^{n}a_i)^p\le n^{p-1}\sum_{i=1}^{n}a_i^p$ for $p\ge2$.
Applying Lemma~\ref{localtimelp} to $Y_t-Y^{m}_t $ and $Y^{m,n}_t-Y^{m}_t$ and following the same broad strategy as \cite{de_angelis_numerical_2022} allows to prove again Lemmas~\ref{Lp:stability} and \ref{lp:numerical} and the resulting $L^p$-sup convergence rate of Theorem~\ref{rate:Lp} without the Yamada-Watanabe approximation technique.

\section{Proof of Theorem~\ref{rate:L1}: $L^1$-sup Convergence}
\label{sec:l1}
We now adapt the arguments Section~\ref{sec:lp} to obtain convergence in the
$L^1$-$\sup$ norm, the broad strategy is the same, but while in both Lemma~\ref{Lp:stability} and Lemma~\ref{lp:numerical} the error bounds took a form that allowed the use of Lemma~\ref{GR_Lemma32}, in the analogous $L^1$-$\sup$ that follow this will not be available, to circumvent this Young's inequality and the $L^1$ results of
\cite{chaparro_jaquez_convergence_2025} and \cite{de_angelis_numerical_2022} will be used.

In both Lemmas~\ref{Lp:stability} and \ref{lp:numerical}, the final convergence bounds are proven applying Lemma~\ref{GR_Lemma32} to the integral terms appearing in \eqref{Lp:stability:line2} and \eqref{Lp:numerical:line3}. Following the same steps in the case of the $L^1$-sup error will give the same type of integral terms, but the coefficients will not satisfy the hypotheses of Lemma~\ref{GR_Lemma32}. To bound those terms, we follow the method employed in the proof of \cite[Theorem 2.11]{ngo_strong_2016} to get a rate similar to the one obtained in the $L^1$ case of \cite{de_angelis_numerical_2022} and \cite{chaparro_jaquez_convergence_2025}, but with an additional power of $2\alpha - 1$ in both sides of the decomposition.
\begin{lemma}[$L^1$-$\sup$ Stability Estimate]
    \label{l1:stability}
    Let Assumption \ref{hypothesis:b} hold and fix $\hat{\beta}\in(\beta,\tfrac12)$ as above and $\alpha\in(\tfrac12,1-\hat{\beta})$. Then, for all $m\in\mathbb{N}^+$ such that
    $\|b^m - b\|_{C_T\mathcal{C}^{-\hat{\beta}}}<1$ there exists a constant $C>0$ such that,
    \begin{equation*}
    \mathbb{E}\!\left[\sup_{0 \leq t \leq T}\left\lvert X^m_t-X_t \right\rvert\right]
    \leq C \, \left\lVert b^{m}-b \right\rVert_{C_T\mathcal{C}^{-\hat{\beta}}}^{(2\alpha-1)^2}.
    \end{equation*}
\end{lemma}
\begin{proof}
The beginning of the proof is the same as the one for Lemma~\ref{Lp:stability}, starting with estimating the difference between $X^{m}$ and $X$ via the associated processes $Y^{m}$ and $Y$.
\begin{align}
    \mathbb{E}\!\left[\sup_{0\leq t\leq T}\lvert X^{m}_t - X_t\rvert\right] &= \mathbb{E}\!\left[\sup_{0\leq t\leq T}\lvert \psi^{m}(t,Y^{m}_t)-\psi\!\left(t,Y_t\right)\rvert\right]\nonumber\\
    &\leq \mathbb{E}\!\left[\sup_{0\leq t\leq T}\lvert \psi^{m}(t,Y^{m}_t)-\psi\!\left(t,Y^{m}_t\right)\rvert\right]
     + \mathbb{E}\!\left[\sup_{0\leq t\leq T}\lvert \psi\!\left(t,Y^{m}_t\right)-\psi\!\left(t,Y_t\right)\rvert\right]\nonumber\\
    &\leq 2c\lVert b^{m}-b\rVert_{C_T\mathcal{C}^{-\hat{\beta}}} + 2\mathbb{E}\!\left[\sup_{0\leq t\leq T}\lvert Y^{m}_t-Y_t\rvert\right],\label{splitstab}
\end{align}
where we used Lemma~\ref{jaquez:L14} in the second inequality. It remains to estimate the second term to obtain the desired stability bound. We have, as in Lemma~\ref{Lp:stability},
\begin{align}
    &\mathbb{E}\!\left[\sup_{0\leq t\leq T}\lvert Y^{m}_t-Y_t\rvert\right]\leq \mathbb{E}\!\left[\sup_{0\leq t\leq T}\bigl(\kappa+\varphi_{\delta,\kappa}(Y^{m}_t-Y_t)\bigr)\right]\leq \kappa\nonumber\\
    &\quad+\mathbb{E}\!\left[\sup_{0\leq t\leq T}\int_{0}^{t}\varphi'_{\delta,\kappa}(Y^{m}_s-Y_s)\Big(u^{m}\bigl(s,\psi^{m}(s,Y^{m}_s)\bigr)-u\!\left(s,\psi(s,Y_s)\right)\Big)\,ds\right]\label{L1:stability:line1}\\
    &\quad+\mathbb{E}\!\left[\sup_{0\leq t\leq T}\int_{0}^{t}\varphi'_{\delta,\kappa}(Y^{m}_s-Y_s)\Big(u_x^{m}(s,\psi^{m}(s,Y^{m}_s))-u_x\!\left(s,\psi(s,Y_s)\right)\Big)\,dW_s\right]\label{L1:stability:line2}\\
    &\quad+\mathbb{E}\!\left[\sup_{0\leq t\leq T}\int_{0}^{t}\varphi''_{\delta,\kappa}(Y^{m}_s-Y_s)\Big(u_x^{m}(s,\psi^{m}(s,Y^{m}_s))-u_x\!\left(s,\psi(s,Y_s)\right)\Big)^2\,ds\right].\label{L1:stability:line3}
\end{align}
Using Lemma~\ref{jaquez:L14} and $|\varphi'_{\delta,\kappa}|\le1$, we get
\begin{align*}
\eqref{L1:stability:line1}
&\leq \mathbb{E}\!\left[\sup_{0\leq t\leq T}\int_{0}^{t}\bigl(2c\lVert b^{m}-b\rVert_{C_T\mathcal{C}^{-\hat{\beta}}}+\lvert Y^{m}_s-Y_s\rvert\bigr)\,ds\right]\\
&\leq 2Tc\lVert b^{m}-b\rVert_{C_T\mathcal{C}^{-\hat{\beta}}}+\mathbb{E}\!\left[\int_{0}^{T}\sup_{0\leq r\leq T}\lvert Y^{m}_r-Y_r\rvert\,ds\right].
\end{align*}
For \eqref{L1:stability:line3}, using \eqref{propphi:C} and Lemma~\ref{jaquez:L14}:
\begin{align*}
\eqref{L1:stability:line3}
&\leq \mathbb{E}\!\Bigg[
    \sup_{0\leq t\leq T}
    \int_{0}^{t}
    \Big(
        C\lVert b^{m}-b\rVert_{C_T\mathcal{C}^{-\hat{\beta}}}
\\
&\qquad\qquad
        + 2^\alpha \lVert u_x\rVert_{C_T\mathcal{C}^{1+\alpha}}
          C^\alpha \lVert b^{m}-b\rVert_{C_T\mathcal{C}^{-\hat{\beta}}}^{2\alpha-1}
\\
&\qquad\qquad
        + 2^\alpha \lVert u_x\rVert_{C_T\mathcal{C}^{1+\alpha}}
          \lvert Y^{m}_s - Y_s\rvert^\alpha
    \Big)
    \varphi''_{\delta,\kappa}(Y^{m}_s-Y_s)\,ds
\Bigg]
\\
&\leq C\frac{2}{\kappa\log\delta}\lVert b^{m}-b\rVert_{C_T\mathcal{C}^{-\hat{\beta}}}^2
  + C\frac{2}{\kappa\log\delta}\lVert b^{m}-b\rVert_{C_T\mathcal{C}^{-\hat{\beta}}}^{2\alpha}
  + C\frac{2}{\log\delta}\kappa^{2\alpha-1}.
\end{align*}
Choosing $\kappa=\lVert b^{m}-b\rVert_{C_T\mathcal{C}^{-\hat{\beta}}}$ and fixing $\delta>1$ gives
\[
    \eqref{L1:stability:line3}\leq C\bigg(\lVert b^{m}-b\rVert_{C_T\mathcal{C}^{-\hat{\beta}}} + \lVert b^{m}-b\rVert_{C_T\mathcal{C}^{-\hat{\beta}}}^{2\alpha-1}\bigg).
\]
For the martingale term \eqref{L1:stability:line2}, BDG implies
\begin{align*}
    \eqref{L1:stability:line2}
    &\leq C\mathbb{E}\!\left[\Big(\int_{0}^{T}\Big(u_x^{m}(s,\psi^{m}(s,Y^{m}_s))-u_x(s,\psi(s,Y_s))\Big)^2ds\Big)^{1/2}\right]\\
    &\leq C\lVert b^{m}-b\rVert_{C_T\mathcal{C}^{-\hat{\beta}}}+C\lVert b^{m}-b\rVert_{C_T\mathcal{C}^{-\hat{\beta}}}^\alpha
      +\Big(\mathbb{E}\!\big[\int_{0}^{T}\lvert Y^{m}_s-Y_s\rvert^{2\alpha}ds\big]\Big)^{1/2}.
\end{align*}
Notice that this bound does not satisfy the hypotheses of Lemma~\ref{GR_Lemma32}, thus as said before, we use the methods used to prove \cite[Theorem 2.11]{ngo_strong_2016} and split the integral appearing in the last bound using the Young inequality $xy\leq\frac{x^2}{2C} + \frac{Cy^2}{2}$:
\begin{align*}
    \Big(\mathbb{E}\!\big[\int_{0}^{T}\lvert Y^{m}_s-Y_s\rvert^{2\alpha}ds\big]\Big)^{1/2}
    &\leq C\mathbb{E}\!\left[\Big(\sup_{0\leq t\leq T}\lvert Y^{m}_t-Y_t\rvert^{1/2}\Big)\Big(\int_{0}^{T}\lvert Y^{m}_s-Y_s\rvert^{2\alpha-1}ds\Big)^{1/2}\right]\\
    &\leq \frac{1}{2}\mathbb{E}\!\left[\sup_{0\leq t\leq T}\lvert Y^{m}_t-Y_t\rvert\right]
      +\frac{C}{2}\Big(\int_{0}^{T}\mathbb{E}\!\left[\lvert Y^{m}_s-Y_s\rvert\right]ds\Big)^{2\alpha-1}\\
      &\leq\frac{1}{2}\mathbb{E}\!\left[\sup_{0\leq t\leq T}\lvert Y^{m}_t-Y_t\rvert\right]
      +C\lVert b^{m}-b\rVert_{C_T\mathcal{C}^{-\hat{\beta}}}^{(2\alpha-1)^2}.
\end{align*}
Having used \ref{jaquez:P6} in the last inequality.

Collecting all terms and absorbing the $\frac{1}{2}\mathbb{E}[\sup|Y^m-Y|]$ into the left-hand side then gives
\[
    \mathbb{E}\!\left[\sup_{0\leq t\leq T}\lvert Y^{m}_t-Y_t\rvert\right]\leq C\bigg(\lVert b^{m}-b\rVert_{C_T\mathcal{C}^{-\hat{\beta}}}+\lVert b^{m}-b\rVert_{C_T\mathcal{C}^{-\hat{\beta}}}^{(2\alpha-1)^2}\bigg).
\]
Since by hypothesis we have that $\lVert b^{m}-b\rVert_{C_T\mathcal{C}^{-\hat{\beta}}}\leq 1$, the slowest converging term will be the one with the $2\alpha-1$ exponent since $\alpha<1-\beta$ implies $2\alpha-1< 1-2\beta<1$.
Together with \eqref{splitstab}, absorbing all higher powers into the leading constant gives the statement.
\end{proof}
\begin{lemma}[$L^1$-$\sup$ Numerical Estimate]
    \label{l1:numerical}
    Let Assumption \ref{hypothesis:b} hold and fix $\alpha\in(\tfrac12,1-\hat{\beta})$. Then
   \begin{equation*}
   \mathbb{E}\!\left[\sup_{0 \leq t \leq T}\left\lvert X^{m,n}_t-X^m_t \right\rvert\right]\leq C(A_m\,n^{-1}+ B_m\,n^{-\frac{1}{2}})^{2\alpha - 1},
   \end{equation*}
    where $A_m$ and $B_m$ are the same as in Lemma~\ref{DGI:Lemma62} and $C>0$ does not depend on $m$ or $n$.
\end{lemma}
\begin{proof}
    Recall that 
    \begin{equation}
        \label{l1:numerical:star}
        \mathbb{E}\!\left[\sup_{0\le t\le T}\left|X^{m,n}_t-X^m_t\right|\right]
\leq C\,\mathbb{E}\!\left[\sup_{0\le t\le T}\left|Y^{m,n}_t-Y^{m}_t\right|\right].
    \end{equation}
 As in the $L^p$-sup case, we bound $|Y^{m,n} - Y^{m}|$ with $\kappa+\varphi_{\delta,\kappa}$ using the It\^o Formula:
\begin{align}
    &\mathbb{E}\!\left[\sup_{0\leq t\leq T}\left\lvert Y^{m,n}_t-Y^{m}_t\right\rvert\right]
    \leq \mathbb{E}\!\left[\sup_{0\leq t\leq T}\left(\kappa+\varphi_{\delta,\kappa}\bigl(Y^{m,n}_t-Y^{m}_t\bigr)\right)\right]\leq \kappa
    \nonumber\\
    &\quad + \mathbb{E}\!\left[\sup_{0\leq t\leq T}\int_{0}^{t}\varphi'_{\delta,\kappa}\bigl(Y^{m,n}_s-Y^{m}_s\bigr)\bigl(u^{m}(s,\psi^{m}(s,Y^{m,n}_s))-u^{m}(s,\psi^{m}(s,Y^{m}_s))\bigr)\,ds\right]\label{L1:numerical:line1}\\
    &\quad + \mathbb{E}\!\left[\sup_{0\leq t\leq T}\int_{0}^{t}\varphi'_{\delta,\kappa}\bigl(Y^{m,n}_s-Y^{m}_s\bigr)
        \bigl(b^{m}\bigl(t_{k(s)},X_{t_{k(s)}}^{m,n}\bigr)-b^{m}\bigl(s,X_s^{m,n}\bigr)\bigr)\bigl(1+u^{m}_x(s,X_s^{m,n})\bigr)\,ds\right]\label{L1:numerical:line2}\\
    &\quad + \mathbb{E}\!\left[\sup_{0\leq t\leq T}\int_{0}^{t}\varphi'_{\delta,\kappa}\bigl(Y^{m,n}_s-Y^{m}_s\bigr)
        \bigl(u^{m}_x(s,\psi^{m}(s,Y^{m,n}_s))-u^{m}_x(s,\psi^{m}(s,Y^{m}_s))\bigr)\,dW_s\right]\label{L1:numerical:line3}\\
    &\quad + \mathbb{E}\!\left[\sup_{0\leq t\leq T}\int_{0}^{t}\varphi''_{\delta,\kappa}\bigl(Y^{m,n}_s-Y^{m}_s\bigr)
        \bigl(u^{m}_x(s,\psi^{m}(s,Y^{m,n}_s))-u^{m}_x(s,\psi^{m}(s,Y^{m}_s))\bigr)^2\,ds\right].\label{L1:numerical:line4}
\end{align}
We deal with each term separately, starting from \eqref{L1:numerical:line1}, where we use that $|\phi_{\delta,\kappa}|\leq1$ and that $u^m\circ \psi^m$ is Lipschitz: 
\begin{align*}
    \eqref{L1:numerical:line1}
    &\leq \mathbb{E}\!\left[\sup_{0\leq t\leq T}\int_{0}^{t} C\left\lvert Y^{m,n}_s-Y^{m}_s\right\rvert \,ds\right]
    \leq C\int_{0}^{T}\mathbb{E}\!\left[\sup_{0\leq r\leq T}\left\lvert Y^{m,n}_r-Y^{m}_r\right\rvert\right]\,ds.
\end{align*}
For \eqref{L1:numerical:line4}, using \eqref{propphi:C} and the spatial $\alpha$-H\"older regularity of $u^m_x$,
\begin{align*}
    \eqref{L1:numerical:line4}
    &\leq C\mathbb{E}\!\left[\sup_{0\leq t\leq T}\int_{0}^{t}
        \frac{2}{\left\lvert Y^{m,n}_s-Y^{m}_s\right\rvert\log\delta}\,
        \mathbbm{1}_{\{Y^{m,n}_s-Y^{m}_s\in\left[\frac{\kappa}{\delta},\kappa\right]\}}\cdot
        \left\lvert Y^{m,n}_s-Y^{m}_s\right\rvert^{2\alpha}\,ds\right]\\
    &\leq C\frac{2}{\log\delta}\,\kappa^{2\alpha-1}\,T.
\end{align*}
For the drift term we use Lemma~\ref{DGI:Lemma62}:
\begin{align*}
    \eqref{L1:numerical:line2}
    &\leq \frac{3}{2}\,\mathbb{E}\!\left[\int_{0}^{T}\left\lvert b^{m}\bigl(t_{k(s)},X_{t_{k(s)}}^{m,n}\bigr)-b^{m}\bigl(s,X_s^{m,n}\bigr)\right\rvert ds\right]\leq A_m\, n^{-1} + B_m\, n^{-\frac{1}{2}}.
\end{align*}
While for the martingale term we apply BDG and then Young's inequality:
\begin{align*}
    \eqref{L1:numerical:line3}
    &\leq C\,\mathbb{E}\!\left[\Big(\int_{0}^{T}\left\lvert Y^{m}_s-Y^{m,n}_s\right\rvert^{2\alpha} ds\Big)^{1/2}\right]\\
    &\leq C\,\mathbb{E}\!\left[\Big(\sup_{0\leq t\leq T}\left\lvert Y^{m}_t-Y^{m,n}_t\right\rvert^{1/2}\Big)
        \Big(\int_{0}^{T}\left\lvert Y^{m}_s-Y^{m,n}_s\right\rvert^{2\alpha-1} ds\Big)^{1/2}\right]\\
    &\leq \frac{1}{2}\,\mathbb{E}\!\left[\sup_{0\leq t\leq T}\left\lvert Y^{m}_t-Y^{m,n}_t\right\rvert\right]
    +\frac{C}{2}\left(\int_{0}^{T}\mathbb{E}\!\left[\left\lvert Y^{m}_s-Y^{m,n}_s\right\rvert\right] ds\right)^{2\alpha-1}\\
    &\leq\frac{1}{2}\,\mathbb{E}\!\left[\sup_{0\leq t\leq T}\left\lvert Y^{m}_t-Y^{m,n}_t\right\rvert\right] + C\bigg( A_m\, n^{-1} + B_m\, n^{-\frac{1}{2}}\bigg)^{2\alpha - 1}.
\end{align*}
Using again the $L^1$ estimate of Lemma~\ref{DGI:Lemma62}.

Collecting all terms, absorbing the $\frac{1}{2}\mathbb{E}[\sup|Y^{m,n}-Y^{m}|]$ to the left-hand side, and recalling \eqref{l1:numerical:star} gives the statement by taking the limit for $\kappa\to0$.
\end{proof}

\begin{proof}[Proof of Theorem \ref{rate:L1}]
    From Lemma~\ref{l1:stability} and Lemma~\ref{l1:numerical} we have, for a suitable constant $C>0$,
    \begin{equation}\label{stabnum}
      \mathbb{E}\Big[\sup_{0\le t\le T}\bigl|X^m_t-X_t\bigr|\Big]
      \le C\,\|b^m-b\|_{C_T\mathcal{C}^{-\hat{\beta}}}^{(2\alpha-1)^2},
    \end{equation}
    and
    \begin{equation}\label{numnum}
      \mathbb{E}\Big[\sup_{0\le t\le T}\bigl|X^{m,n}_t-X^m_t\bigr|\Big]
      \le C\,T\Big(A_m n^{-1} + B_m n^{-\frac{1}{2}}\Big)^{2\alpha-1},
    \end{equation}
By our mollification choice $b^m = P_{1/m}b$ and the Schauder estimates for the heat semigroup (Lemma~\ref{Schauder}), we gain spatial regularity from $-\beta$ up to $-\hat{\beta}$. More precisely, for our fixed $\hat{\beta}\in(\beta,\frac{1}{2})$ there exists $C>0$ such that
    \[
      \|b^m-b\|_{C_T\mathcal{C}^{-\hat{\beta}}}
      \le C\, m^{-\frac{\hat{\beta}-\beta}{2}}\|b\|_{C_T\mathcal{C}^{-\beta}}.
    \]
    The exponent $\frac{\hat{\beta}-\beta}{2}$ reflects the gain in spatial regularity from $-\beta$ to $-\hat{\beta}$ by heat-kernel smoothing.
    Introducing this estimate, we see that the stability error \eqref{stabnum} can be bound by
    \[
      n^{-\eta\frac{\hat{\beta}-\beta}{2}(2\alpha-1)^2},
    \]
by letting $m=m(n) = n^\eta$.

On the other hand, with the same $m$, the slowest term in the numerical error bound \eqref{numnum} is
\[
  n^{\left(\eta\frac{\varepsilon+\hat{\beta}+1}{2}-\frac{1}{2}\right)(2\alpha-1)},
\]
so to optimise the rate we balance these two exponents, i.e.
\begin{equation}\label{eq:balance-exponents}
  -\eta\frac{\hat{\beta}-\beta}{2}(2\alpha-1)^2
  = \left(\eta\frac{\varepsilon+\hat{\beta}+1}{2}-\frac{1}{2}\right)(2\alpha-1).
\end{equation}
Following \cite{chaparro_jaquez_convergence_2025}, we write $\alpha\in(0,\frac12)$ as $\alpha = 1-\hat{\beta}-\varepsilon$, 
with $\varepsilon\in(0,1 - \beta)$ and $\hat{\beta}\in(\beta,\tfrac12)$. Solving \eqref{eq:balance-exponents} for $\eta$ then yields
\begin{equation}\label{eq:eta-opt}
  \eta = \frac{1}{2\left[\frac{\varepsilon+\hat{\beta}+1}{2}
        +\left(\frac{1}{2}-\hat{\beta}-\varepsilon\right)^2\right]}.
\end{equation}
Substituting this choice of $\eta$ into the exponents of $m(n)$ shows that
\[
  \mathbb{E}\left[\sup_{0\leq t\leq T}\bigl|X_t^{m(n),n}-X_t\bigr|\right]
  \leq C\, n^{-r(\beta,\varepsilon)\left(\frac{1}{2}-\beta-\varepsilon\right)},
\]
concluding the proof.
\end{proof}
\section*{Acknowledgements}
The author is a member of INdAM-GNAMPA. The present research was first presented at the INdAM Workshop on Irregular Stochastic Analysis in June 2025.
\printbibliography
\end{document}